\documentclass[11pt,reqno]{amsart}
\usepackage{tikz}
\usepackage{subfig}
\usepackage{epstopdf}
\usepackage{epsfig}
\usepackage{math}
\graphicspath{{./figures/}}
\usepackage{tikz}
\usetikzlibrary{shapes,arrows}
\usepackage{adjustbox}
\usepackage{rotating}
\usepackage{graphicx}

\usepackage{booktabs, multirow} % for borders and merged ranges
\usepackage{soul}% for underlines

\tikzstyle{decision} = [diamond, draw, fill=blue!20,
    text width=4.5em, text badly centered, node distance=3cm, inner sep=0pt]
\tikzstyle{block} = [rectangle, draw, fill=blue!20,
    text width=5em, text centered, rounded corners, minimum height=4em]
\tikzstyle{line} = [draw, -latex']
\tikzstyle{cloud} = [draw, ellipse,fill=red!20, node distance=3cm,
    minimum height=2em]

\usetikzlibrary{positioning}
\tikzset{main node/.style={circle,fill=blue!20,draw,minimum size=1cm,inner sep=0pt},  }

\begin{document}
\title[]{A kernel formula for regularized Wasserstein proximal operators}
\author[Li]{Wuchen Li}
\email{wuchen@mailbox.sc.edu}
\address{Department of Mathematics, University of South Carolina, Columbia }
\author[Liu]{Siting Liu}
\email{siting6@math.ucla.edu}
\address{Department of Mathematics, University of California, Los Angeles}
\author[Osher]{Stanley Osher}
\email{sjo@math.ucla.edu}
\address{Department of Mathematics, University of California, Los Angeles}
\newcommand{\vr}{\overrightarrow}
\newcommand{\wt}{\widetilde}
\newcommand{\dd}{\mathcal{\dagger}}
\newcommand{\ts}{\mathsf{T}}
\thanks{ W. Li is supported by  AFOSR MURI FA9550-18-1-0502,  AFOSR YIP award 2023, and NSF RTG: 2038080. S. Liu and S. Osher thank the funding from AFOSR MURI FA9550-18-1-0502 and ONR grants: N00014-18-1-2527, N00014-20-1-2093, and N00014-20-1-2787.}
\maketitle
\begin{abstract}
We study a class of regularized proximal operators in Wasserstein-2 space. We derive their solutions by kernel integration formulas. We obtain the Wasserstein proximal operator using a pair of forward-backward partial differential equations consisting of a continuity equation and a Hamilton-Jacobi equation with a terminal time potential function and an initial time density function. 
Following \cite{HFO, OHF}, we regularize the PDE pair by adding forward and backward Laplacian operators. We apply Hopf-Cole type transformations to rewrite these regularized PDE pairs into forward-backward heat equations. We then use the fundamental solution of the heat equation to represent the regularized Wasserstein proximal with kernel integral formulas. Numerical examples show the effectiveness of kernel formulas in approximating the Wasserstein proximal operator.  
\end{abstract}

\section{Introduction}
Proximal operators are essential tools in optimization. Recently, proximal operators in probability space equipped with the Wasserstein-2 metric have shown to be useful in scientific computing \cite{LLW} and machine learning \cite{B, WGAN}. The proximal operator in Wasserstein-2 space is called the Wasserstein proximal operator \cite{AGS}. An interesting example is the Wasserstein proximal operator of the Kullback--Leibler (KL) divergence, known as the JKO (Jordan-Kinderlehrer-Otto) scheme \cite{JKO}. This has been used to approximate the Wasserstein gradient flow as a backward Euler scheme in time discretizations.  

However, computing the Wasserstein proximal operator requires an optimization procedure. One has to develop an optimization step to compute or approximate Wasserstein metrics and energy functionals. 
This paper proposes an alternative approach to approximate the Wasserstein proximal operator. We use an optimal control formulation of the Wasserstein proximal operator, whose minimizer forms a pair of forward-backward partial differential equations, including continuity and Hamilton-Jacobi equations. We add forward-backward Laplacian operators into the PDE system as a regularization. Applying Hopf-Cole type transformations enables us to rewrite the PDE system into forward-backward heat equations. Using the fundamental solution of two heat equations, we write the regularized Wasserstein proximal of linear energy explicitly with a simple kernel formula. We apply the same idea to the regularized Wasserstein proximal of nonlinear energy. This turns into a nonlinear integral equation, which can be solved by simple fixed-point iteration methods. 

We now present the approximation of Wasserstein proximal of linear energy below. Given a probability density function $\rho_0$ with the finite second moment, a bounded potential function $V\in C^{1}(\mathbb{R}^d)$, and a scalar constant $T>0$, consider the Wasserstein proximal operator of the linear energy below: 
\begin{equation*}
\rho_T:=\mathrm{WProx}_{T\mathcal{V}}(\rho_0):=\arg\min_{q\in \mathcal{P}_2(\mathbb{R}^d)} \quad \int_{\mathbb{R}^d} V(x)q(x)dx+\frac{\mathcal{W}(\rho_0, q)^2}{2T}. 
\end{equation*}
 where $\mathcal{W}(\rho_0, q)$ is the Wasserstein-2 distance between $\rho_0$ and $q$, and the minimization is taken among all probability density functions $q$ with the finite second moment. We denote the minimizer $\rho_T=q$ as the Wasserstein proximal operator. According to the Benamou-Brenier formula \cite{BB} and simple derivations in section \ref{sec2}, the Wasserstein proximal operator forms an optimal control problem, whose minimizer satisfies a forward-backward PDE system:
  \begin{equation}\label{PDE_BB}
\left\{\begin{aligned}
&\partial_t\rho(t,x)+\nabla\cdot\big(\rho(t,x)\nabla_x\Phi(t,x)\big)=0,\\
&\partial_t\Phi(t,x)+\frac{1}{2}\|\nabla_x\Phi(t,x)\|^2=0,\\
& \rho(0,x)=\rho_0(x),\quad \Phi(T,x)=-V(x). 
\end{aligned}\right.
\end{equation}
Here $t\in [0, T]$, and $\rho_T$ satisfies the Wasserstein proximal operator of linear energy. 

In this paper,  {motivated by regularized proximal operators in Euclidean space \cite{HFO, OHF},} we study a regularized PDE system of equation \eqref{PDE_BB}. For a constant scalar $\beta>0$,  we have
\begin{equation}\label{PDE_BB_1}
\left\{\begin{aligned}
&\partial_t\rho(t,x)+\nabla\cdot\big(\rho(t,x)\nabla_x\Phi(t,x)\big)=\beta \Delta_x \rho(t,x),\\
&\partial_t\Phi(t,x)+\frac{1}{2}\|\nabla_x\Phi(t,x)\|^2=-\beta \Delta_x\Phi(t,x),\\
& \rho(0,x)=\rho_0(x),\quad \Phi(T,x)=-V(x). 
\end{aligned}\right.
\end{equation}
By applying Hopf-Cole type transformations, we demonstrate that equation \eqref{PDE_BB_1} has the solution below:
\begin{equation*}
\left\{\begin{aligned}
&\rho(t,x)=(G_{T-t}*e^{-\frac{V}{2\beta}})(x)\cdot (G_t*\frac{\rho_0}{G_T*e^{-\frac{V}{2\beta}}})(x),\\
&\Phi(t,x)=2\beta \log (G_{T-t}*e^{-\frac{V}{2\beta}})(x). 
\end{aligned}\right.
\end{equation*}
where $*$ is a convolution operator and $G_t$ is the scaled heat kernel denoted as
\begin{equation*}
G_t(x,y):=\frac{1}{\sqrt{(4\pi \beta t)^d}}e^{-\frac{\|x-y\|^2}{4\beta t}}.
\end{equation*}
In particular, we obtain an integral representation for the terminal density function:
\begin{equation}\label{kernel_formula}
\rho(T,x)=\int_{\mathbb{R}^d}K(x,y)\rho(0,y)dy,
\end{equation}
where $K\colon \mathbb{R}^d\times\mathbb{R}^d\rightarrow\mathbb{R}$ is a kernel function 
\begin{equation*}
K(x,y):=\frac{e^{-\frac{1}{2\beta}(V(x)+\frac{\|x-y\|^2}{2T})}}{\int_{\mathbb{R}^d}e^{-\frac{1}{2\beta}(V(z)+\frac{\|z-y\|^2}{2T})}dz}.
\end{equation*}
For a fixed variable $y\in\mathbb{R}^d$, the kernel function $K(x,y)$ is a Gibbs distribution or softmax function of variable $x$ with constant $\frac{1}{2\beta}$ and function $V(x)+\frac{1}{2T}\|x-y\|^2$. 
Formula \eqref{kernel_formula} is a kernel integration representation of the regularized Wasserstein proximal operator. It could be useful in approximating or computing high-dimensional sampling problems; see previous studies in \cite{HFO, OHF}. 

Wasserstein proximal operators have been widely studied \cite{AGS}. The regularized Wasserstein proximal dynamics \eqref{PDE_BB_1} and their Hopf-Cole type transformations \cite{Nelson2} have been widely studied in Schr{\"o}dinger bridge systems \cite{CGP,Yasue1981_stochastica,Zambrini1986} and Schr{\"o}dinger equations \cite{BGU,Nelson2}; see also generalized Hopf-Cole type transformations \cite{LeL}. 
In particular, there are Schr{\"o}dinger-Follermer diffusions \cite{fol88}, which are closely related to the equation system \eqref{PDE_BB_1}. We also remark that the optimal control formulation of the Wasserstein-proximal operator is a variational problem in potential mean-field games; see \cite{LL}. It has been used in studying and computing particular classes of mean-field control and mean-field games \cite{OG, OG1}. Compared to all the above studies, we study Schr{\"o}dinger bridge systems with different boundary conditions. It contains a given initial time density function and a terminal time potential function, different from fixed initial and terminal time density functions as in Schr{\"o}dinger bridge systems. This allows us to obtain a closed-form kernel integration update, at least for regularized Wasserstein proximal operators of linear energies. For nonlinear potential energies, our kernel formula also provides an iterative update to compute Wasserstein proximal operators. 

The paper is organized as follows. We introduce Wasserstein proximal operators of linear energies, their diffusion regularizations, and minimization systems in section \ref{sec2}. In section \ref{sec3}, we derive the kernel representation of regularized Wasserstein proximal of linear energies. Finally, we generalize the kernel formulation for Wasserstein proximal operators of nonlinear energies in section \ref{sec4}. Numerical experiments of proposed kernel functions are presented in section \ref{section5}. 
\section{Regularized Wasserstein proximal operators of linear energies}\label{sec2} 
In this section, we study a class of mean-field control problems that lead to regularized proximal operators of linear energies in Wasserstein-2 space. We derive the minimizing equations, which is a forward-backward PDE system.  
\subsection{Wasserstein proximal operators of linear energies}
Consider a linear energy functional 
\begin{equation*}
\mathcal{V}(\rho)=\int_{\mathbb{R}^d} V(x)\rho(x)dx,
\end{equation*}
where $V\in C^1(\mathbb{R}^d)$ is a given function. Denote a terminal time $T>0$ and $\rho_0\in \mathcal{P}_2(\mathbb{R}^d)$, where $\mathcal{P}_2(\mathbb{R}^d)$ is the probability density set with finite second moment. Consider the proximal operator of a linear energy in Wasserstein-2 space.
\begin{equation}\label{Wprox}
\rho_T:=\mathrm{WProx}_{T\mathcal{V}}(\rho_0):=\arg\min_{q\in \mathcal{P}_2(\mathbb{R}^d)} \quad \mathcal{V}(q)+\frac{\mathcal{W}(\rho_0, q)^2}{2T}. 
\end{equation}
Here $\mathcal{W}$ denotes the Wasserstein-2 distance between two densities $\rho^0$ and $\rho$; see \cite{AGS, Villani2009_optimal}. Recall that the Wasserstein-2 distance can be recast as an optimal control problem, known as Benamou-Brenier formula \cite{BB}. Consider
\begin{equation*}
\frac{\mathcal{W}(\rho_0, q)^2}{2T}:=\inf_{\rho, v, \rho_T} \quad \int_0^T\int_{\mathbb{R}^d}\frac{1}{2}\|v(t,x)\|^2\rho(t,x)dxdt,
\end{equation*}
where the minimizer is taken among all vector fields $v\colon [0, T]\times\mathbb{R}^d\rightarrow\mathbb{R}^d$, density functions $\rho\colon [0, T)\times \mathbb{R}^d\rightarrow\mathbb{R}$, such that 
\begin{equation*}
\partial_t\rho(t,x)+\nabla\cdot(\rho(t,x)v(t,x))=0,\quad \rho(0,x)=\rho_0(x), \quad \rho(T,x)=q(x).
\end{equation*}
Wasserstein proximal operators are useful in designing sampling algorithms. However, the Wasserstein proximal operator \eqref{Wprox} does not have a closed-form update. One needs to apply numerical methods to compute the variational problem \eqref{Wprox}. 
\subsection{Regularized Wasserstein proximal operators of linear energies}
In this subsection, we consider a regularized Wasserstein proximal operator, which adds a diffusive term to the continuity equation of the Benamou-Brenier formula. This regularization forms an interesting mean-field control problem. We then derive the related minimizing system with initial-terminal conditions. In the next section, we demonstrate that this minimizing system has a closed-form kernel representation.

\begin{definition}\label{def1}
Let $\beta>0$ be a scalar, and $\|\cdot\|$ is the Euclidean norm. Consider the following mean-field control problem:
\begin{equation}\label{variation}
\inf_{\rho, v, q} \quad \int_0^T\int_{\mathbb{R}^d}\frac{1}{2}\|v(t,x)\|^2\rho(t,x)dxdt+\int_{\mathbb{R}^d}V(x)q(x)dx, 
\end{equation}
where the minimizer is taken among all drift vector fields $v\colon [0, T]\times\mathbb{R}^d\rightarrow\mathbb{R}^d$, density functions $\rho\colon [0, T)\times \mathbb{R}^d\rightarrow\mathbb{R}$ and terminal time density function $q\colon \mathbb{R}^d\rightarrow\mathbb{R}$, such that 
\begin{equation*}
\partial_t\rho(t,x)+\nabla\cdot(\rho(t,x)v(t,x))=\beta \Delta \rho(t,x),\quad \rho(0,x)=\rho_0(x), \quad \rho(T,x)=q(x).
\end{equation*}

We denote the regularized proximal operator as
\begin{equation*}
\rho_T:=\mathrm{WProx}_{T\mathcal{V},\beta}(\rho_0), 
\end{equation*}
where $\rho_T$ is the solution for the density function in the variational problem \eqref{variation} at time $t=T$. 
\end{definition}
\begin{remark}
We note that when $\beta=0$, variational problem \eqref{variation} is the minimization problem \eqref{Wprox}. We add a diffusion into the continuity equation in the variational problem \eqref{variation}. This forms the regularized Wasserstein proximal operator. 
\end{remark}
We next derive the minimization system for mean-field control problem \eqref{variation}. 
\begin{proposition}[Regularized Wasserstein proximal dynamics]
There exists a Lagrange multiplier function $\Phi\colon [0, T]\times \mathbb{R}^d\rightarrow\mathbb{R}$, such that $(\rho, \Phi)$ satisfies the following forward-backward PDE system. 
\begin{equation}\label{PDE_pair}
\left\{\begin{aligned}
&\partial_t\rho(t,x)+\nabla\cdot\big(\rho(t,x)\nabla_x\Phi(t,x)\big)=\beta \Delta_x \rho(t,x),\\
&\partial_t\Phi(t,x)+\frac{1}{2}\|\nabla_x\Phi(t,x)\|^2=-\beta \Delta_x\Phi(t,x),\\
& \rho(0,x)=\rho_0(x),\quad \Phi(T,x)=-V(x). 
\end{aligned}\right.
\end{equation}
Here the time variable satisfies $t \in [0, T]$, $\rho(t,x)$ is the density function, and $\Phi(t,x)$ is the potential function. Equation system \eqref{PDE_pair} consists of a forward-time Fokker-Planck equation and a backward-time viscous Hamilton-Jacobi equation (Burgers' equation). And the terminal time density function is the regularized Wasserstein proximal density:
\begin{equation*}
  \rho(T,x)=q(x)=\mathrm{WProx}_{T\mathcal{V},\beta}(\rho_0).   
\end{equation*}
\end{proposition}
\begin{proof}
Denote a Lagrange multiplier function $\Phi$. Consider the following saddle point problem 
\begin{equation*}
\inf_{\substack{\rho, v, \rho_T}}\sup_\Phi~\mathcal{L}(\rho,v,\rho_T, \Phi),
\end{equation*}
where
\begin{equation*}
\begin{split}
\mathcal{L}(\rho, v, \rho_T, \Phi):=&\int_0^T \int_{\mathbb{R}^d}\frac{1}{2} \|v(t,x)\|^2\rho(t,x) dx dt+\int_{\mathbb{R}^d}V(x)\rho(T,x)dx\\
&+\int_0^1\int_{\mathbb{R}^d} \Phi(t,x) \Big(\partial_t\rho(t,x)+\nabla\cdot(\rho(t,x) v(t,x))-\beta\Delta\rho(t,x)\Big)dxdt\\
=&\int_0^T \int_{\mathbb{R}^d}\frac{1}{2} \|v(t,x)\|^2\rho(t,x) dx dt+\int_{\mathbb{R}^d}V(x)\rho(T,x)dx\\
&+\int_{\mathbb{R}^d}\Phi(T,x)\rho(T,x)dx-\int_{\mathbb{R}^d}\Phi(0,x)\rho(0,x)dx\\
&+\int_0^1\int_{\mathbb{R}^d} \Big(-\partial_t\Phi(t,x)\rho(t,x)-(\nabla\Phi(t,x), v(t,x)) \rho(t,x)-\beta\Phi(t,x)\Delta\rho(t,x)\Big)dxdt.
\end{split}
\end{equation*}
 The saddle point system satisfies
 \begin{equation*}
 \left\{\begin{aligned}
&\frac{\delta}{\delta v}\mathcal{L}=0,\\
&\frac{\delta}{\delta \Phi}\mathcal{L}=0,\\
&\frac{\delta}{\delta \rho}\mathcal{L}=0,\\
&\frac{\delta}{\delta \rho_T}\mathcal{L}=0.
\end{aligned}\right.
\quad \Rightarrow \quad \left\{\begin{aligned}
&\rho(v-\nabla\Phi)=0,\\
&\partial_t\rho+\nabla\cdot(\rho v)-\beta \Delta\rho=0,\\
&\frac{\|v\|^2}{2}-\partial_t\Phi-\nabla\Phi\cdot v-\beta \Delta\Phi=0\\
&\Phi(T,x)+V(x)=0.
\end{aligned}\right.
 \end{equation*} 
Substituting $v=\nabla\Phi$ into the third equality, we derive two PDEs \eqref{PDE_pair}. Thus, we denote $q(x)=\rho(T,x)$. 
\end{proof}
\begin{remark}
We repeat that $\rho_T=\mathrm{WProx}_{T\mathcal{V},\beta}(\rho_0)$, where $\rho_T(x)=\rho(T,x)$ represents the density function at the terminal time $t=T$ in PDE system \eqref{PDE_pair}. 
\end{remark}
\section{Kernel formulation}\label{sec3}
In this section, we derive $\rho_T$ with a kernel function \eqref{kernel}. We solve a pair of forward-backward PDE system involving a viscous Hamilton-Jacobi equation with a fixed terminal time boundary condition.  By using a Hopf-Cole type transformation, we rewrite the PDEs into forward-backward heat equations with nonlinear initial-terminal time conditions. Thus we solve the original PDE system \eqref{PDE_pair}. 
\begin{proposition}[Hopf-Cole type transformations]
Define a pair of variables $(\eta, \hat \eta)$ as
\begin{equation}\label{CH}
\eta(t,x)=e^{\frac{\Phi(t,x)}{2\beta}},\quad \hat\eta(t,x)=\rho(t,x)e^{-\frac{\Phi(t,x)}{2\beta}}.
\end{equation}
I.e., 
\begin{equation*}
\Phi(t,x)=2\beta\log\eta(t,x),\quad \rho(t,x)=\eta(t,x)\hat\eta(t,x).
\end{equation*}
Then equation \eqref{PDE_pair} in term of $(\eta, \hat \eta)$ satisfies forward-backward heat equations with a nonlinear initial-terminal time boundary condition: 
\begin{equation}\label{heat_system}
\left\{\begin{aligned}
&\partial_t\hat\eta(t,x)=\beta \Delta_x\hat\eta(t,x),\\
&\partial_t\eta(t,x)=-\beta \Delta_x\eta(t,x),\\
& \eta(0,x) \hat\eta(0,x)=\rho_0(x),\quad \eta(T,x)=e^{\frac{\Phi(T,x)}{2\beta}}=e^{-\frac{V(x)}{2\beta}}. 
\end{aligned}\right.
\end{equation}
\end{proposition}
\begin{proof}
The proof follows from classical calculations in Hopf-Cole type transformations. %%We leave its derivation in appendix. 
For completeness of paper, we present the derivation of forwrad-backward heat equations \eqref{heat_system}. Denote $\eta(t,x)=e^{\frac{\Phi(t,x)}{2\beta}}$. Then
\begin{equation*}
\partial_t\eta=\partial_t(e^{\frac{\Phi(t,x)}{2\beta}})=\frac{1}{2\beta}e^{\frac{\Phi}{2\beta}}\partial_t\Phi.
\end{equation*}
And 
\begin{equation*}
\begin{split}
\Delta\eta=&\nabla\cdot(\nabla \eta)=\nabla\cdot(\frac{1}{2\beta}e^{\frac{\Phi}{2\beta}}\nabla \Phi)\\
=&\frac{1}{2\beta} (\nabla e^{\frac{\Phi}{2\beta}}, \nabla\Phi)+\frac{1}{2\beta}e^{\frac{\Phi}{2\beta}}\Delta\Phi\\
=&\frac{1}{(2\beta)^2} e^{\frac{\Phi}{2\beta}}\|\nabla\Phi\|^2+\frac{1}{2\beta}e^{\frac{\Phi}{2\beta}}\Delta\Phi\\
=&\frac{1}{2\beta^2} e^{\frac{\Phi}{2\beta}}\Big(\frac{1}{2}\|\nabla\Phi\|^2+\beta\Delta\Phi\Big).
\end{split}
\end{equation*}
From the viscous Hamilton-Jacobi equation \eqref{PDE_pair}, we have 
\begin{equation*}
\partial_t\eta+\beta \Delta\eta=0.
\end{equation*}
Denote $\hat\eta(t,x)=\rho(t,x)e^{-\frac{\Phi(t,x)}{2\beta}}$. Then 
\begin{equation*}
\begin{split}
\partial_t\hat\eta=& e^{-\frac{\Phi}{2\beta}}\Big\{\partial_t\rho-\rho \frac{1}{2\beta}\partial_t\Phi\Big\}.
\end{split}
\end{equation*}
And 
\begin{equation*}
\begin{split}
\nabla\hat\eta=&\nabla(\rho e^{-\frac{\Phi}{2\beta}})=\nabla\rho e^{-\frac{\Phi}{2\beta}}+\rho \nabla e^{-\frac{\Phi}{2\beta}}\\
=&\nabla\rho e^{-\frac{\Phi}{2\beta}}-\frac{1}{2\beta}\rho \nabla\Phi e^{-\frac{\Phi}{2\beta}}. 
\end{split}
\end{equation*}
Hence 
\begin{equation*}
\begin{split}
\Delta\hat\eta=&\nabla\cdot(\nabla\hat\eta)=\nabla\cdot(\nabla\rho e^{-\frac{\Phi}{2\beta}}-\frac{1}{2\beta}\rho \nabla\Phi e^{-\frac{\Phi}{2\beta}})\\
=&\Delta\rho e^{-\frac{\Phi}{2\beta}}+(\nabla\rho, \nabla e^{-\frac{\Phi}{2\beta}})-\frac{1}{2\beta}(\nabla\rho, \nabla\Phi)e^{-\frac{\Phi}{2\beta}}-\frac{1}{2\beta}\rho \Delta\Phi e^{-\frac{\Phi}{2\beta}}-\frac{1}{2\beta}\rho (\nabla\Phi, \nabla e^{-\frac{\Phi}{2\beta}})\\ 
=&\Delta\rho e^{-\frac{\Phi}{2\beta}}-\frac{1}{\beta}(\nabla\rho, \nabla \Phi) e^{-\frac{\Phi}{2\beta}}-\frac{1}{2\beta}\rho \Delta\Phi e^{-\frac{\Phi}{2\beta}}+\frac{1}{(2\beta)^2}\rho \|\nabla\Phi\|^2\\
=&e^{-\frac{\Phi}{2\beta}}\Big\{\Delta\rho-\frac{1}{\beta}\nabla\cdot(\rho \nabla\Phi)+\frac{1}{2\beta}\rho\Delta\Phi+\frac{1}{4\beta^2}\rho\|\nabla\Phi\|^2\Big\}\\
=&\frac{1}{\beta}e^{-\frac{\Phi}{2\beta}}\Big\{\partial_t\rho -\frac{1}{2\beta}\rho \partial_t\Phi\Big\},
\end{split}
\end{equation*}
where we use the fact $\nabla\cdot(\rho\nabla\Phi)=(\nabla\rho, \nabla\Phi)+\rho\Delta\Phi$. 
Hence we have 
\begin{equation*}
\partial_t\hat\eta-\beta\Delta\hat\eta=0.
\end{equation*}
\end{proof}
We next apply the heat kernel to represent the solution of PDE system \eqref{heat_system}. From equation~\eqref{CH}, we can derive solutions $\rho(t,x)$ and $\Phi(t,x)$ in terms of $\rho_0(x)$ and $V(x)$. 
\begin{proposition}[Kernel solutions]
Denote a scaled heat kernel function: 
\begin{equation*}
G_t(x,y):=\frac{1}{\sqrt{(4\pi \beta t)^d}}e^{-\frac{\|x-y\|^2}{4\beta t}},\quad t\in (0, T].
\end{equation*}
Assume that 
\begin{equation*}
\int_{\mathbb{R}^d}e^{-\frac{1}{2\beta}(V(z)+\frac{\|z-y\|^2}{2T})}dz<+\infty.
\end{equation*}
Then the following solutions hold.
\begin{equation}\label{cfu}
\left\{\begin{aligned}
&\rho(t,x)=(G_{T-t}*e^{-\frac{V}{2\beta}})(x)\cdot (G_t*\frac{\rho_0}{G_T*e^{-\frac{V}{2\beta}}})(x),\\
&\Phi(t,x)=2\beta \log (G_{T-t}*e^{-\frac{V}{2\beta}})(x). 
\end{aligned}\right.
\end{equation}
So
\begin{equation}\label{rho}
\rho(t,x)=\frac{1}{{(4\pi\beta\frac{t(T-t)}{T}})^{\frac{d}{2}}}\int_{\mathbb{R}^d}\int_{\mathbb{R}^d}\frac{e^{-\frac{1}{2\beta}(V(z)+\frac{\|x-z\|^2}{2(T-t)}+\frac{\|x-y\|^2}{2t})}}{\int_{\mathbb{R}^d}e^{-\frac{1}{2\beta}(V(\tilde y)+\frac{\|y-\tilde y\|^2}{2T})}d\tilde y}\rho(0,y) dydz,
\end{equation}
and 
\begin{equation}\label{Phi}
\Phi(t,x)=2\beta \log\Big(\int_{\mathbb{R}^d}\frac{1}{(4\pi\beta (T-t))^{\frac{d}{2}}}e^{-\frac{1}{2\beta}(V(y)+\frac{\|x-y\|^2}{2(T-t)})}dy\Big).
\end{equation}

\end{proposition}
\begin{proof}
Denote $\eta_t(x)=\eta(t,x)$. From $\partial_t\eta+\beta\Delta \eta=0$, we have
\begin{equation}\label{eta}
\begin{aligned}
\eta(t,x)=(G_{T-t}*\eta_T)(x)=(G_{T-t}*e^{-\frac{V}{2\beta}})(x).
\end{aligned}
\end{equation}
When $t=0$, we have
\begin{equation*}
\eta(0,x)=(G_T*\eta_T)(x)=(G_T*e^{-\frac{V}{2\beta}}).
\end{equation*}
From the initial time boundary condition $\rho_0(x)=\eta(0,x)\hat\eta(0,x)$, we have  
\begin{equation*}
\hat\eta(0,x)=\frac{\rho(0,x)}{\eta(0,x)}=\frac{\rho_0(x)}{(G_T*\eta_T)(x)}=\frac{\rho_0(x)}{(G_T*e^{-\frac{V}{2\beta}})(x)}.
\end{equation*}
And from $\partial_t\hat\eta-\beta\Delta\hat\eta=0$, we have 
\begin{equation}\label{heta}
\hat\eta(t,x)=(G_t*\hat\eta_0)(x)=(G_t*\frac{\rho_0}{G_T*e^{-\frac{V}{2\beta}}})(x). 
\end{equation}
Plugging solution \eqref{eta} and \eqref{heta} into Hopf-Cole transformation \eqref{CH}, we obtain 
\begin{equation*}
\rho(t,x)=\eta(t,x)\cdot\hat\eta(t,x)=(G_{T-t}*e^{-\frac{V}{2\beta}})(x)\cdot (G_t*\frac{\rho_0}{G_T*e^{-\frac{V}{2\beta}}})(x),
\end{equation*}
and 
\begin{equation*}
\Phi(t,x)=2\beta\log\eta(t,x) =2\beta \log (G_{T-t}*e^{-\frac{V}{2\beta}})(x). 
\end{equation*}
\end{proof}
In particular, we represent the solution $\rho(T,x)$ by an integral kernel formula. 
\begin{proposition}\label{prop5}
\begin{equation*}
\rho(T,x)=\int_{\mathbb{R}^d}K(x,y, \beta, T, V)\rho(0,y)dy,
\end{equation*}
where $K\colon \mathbb{R}^d\times\mathbb{R}^d\times \mathbb{R}_+\times \mathbb{R}_+\times C^1(\mathbb{R}^d)\rightarrow\mathbb{R}$ is the kernel function 
\begin{equation}\label{kernel}
K(x,y,\beta, T,V):=\frac{e^{-\frac{1}{2\beta}(V(x)+\frac{\|x-y\|^2}{2T})}}{\int_{\mathbb{R}^d}e^{-\frac{1}{2\beta}(V(z)+\frac{\|z-y\|^2}{2T})}dz}.
\end{equation}
\end{proposition}
\begin{proof}
We apply $t=T$ in update \eqref{cfu} for the density function. 
\end{proof}
\begin{remark}
One can check that $\rho(t,x)$ in formula \eqref{rho} is a probability density function. I.e., $\rho(t,x)\geq 0$, and $\int_{\mathbb{R}^d}\rho(t,x)dx=1$. The fact that $\int_{\mathbb{R}^d}\rho(t,x)dx$ is independent of $t$ follows the first equation in \eqref{PDE_pair}. In detail, formula \eqref{rho} can be recast as 
\begin{equation*}\label{rho1}
\rho(t,x)=\frac{1}{{(4\pi\beta\frac{t(T-t)}{T}})^{\frac{d}{2}}}\int_{\mathbb{R}^d}\int_{\mathbb{R}^d}\frac{e^{-\frac{1}{2\beta}(V(z)+\frac{\|y-z\|^2}{2T})}e^{-\frac{1}{2\beta}\frac{\|x-\frac{tz+(T-t)y}{T}\|^2}{2\frac{(T-t)t}{T}}}}{\int_{\mathbb{R}^d}e^{-\frac{1}{2\beta}(V(\tilde y)+\frac{\|y-\tilde y\|^2}{2T})}d\tilde y}\rho(0,y) dydz.
\end{equation*}
From this formula, we directly observe that $\rho(t,x)$ is a probability density function. 
\end{remark}

\begin{remark}
{We remark that $\Phi(t,x)$ in equation \eqref{Phi} is a solution for the viscous Hamilton-Jacobi equation, and $\nabla_x\Phi(0,x)=2\beta\nabla_x\log(G_{T}*e^{-\frac{V}{2\beta}})(x)$ approximates the proximal operator of $V(x)$ in Euclidean space; see \cite{HFO, OHF}. } 
\end{remark}
\begin{remark}
We remark that Hopf-Cole type transformations have been used in studying Schr{\"o}dinger bridge systems \cite{CGP, LeL,Leonard2013_surveya,Yasue1981_stochastica,Zambrini1986}. One can construct Schr{\"o}dinger-F{\"o}llmer diffusion processes \cite{fol88}, which have applications in global optimization problems \cite{D4}. However, we are using different initial-terminal boundary conditions in the usual Schr{\"o}dinger bridge systems with fixed initial and terminal densities. The proposed system \eqref{PDE_pair} has fully closed-form solutions in both density function $\rho$ and potential function $\Phi$. The regularized Wasserstein proximal of linear energies have closed-form representations. Our approach can also work for nonlinear energies. See equations~\eqref{nonlinear}, \eqref{nonlinear_kernel} in the next section. 
\end{remark}
We also present an analytical example of the integral kernel formula in Proposition \ref{prop5}. 
\begin{example}
Consider
\begin{equation*}
V(x)=\frac{1}{2}\|x-x_0\|^2,
\end{equation*}
where $x_0\in \mathbb{R}^d$ is a fixed vector. From \eqref{Phi}, we have 
\begin{equation*}
\begin{split}
\Phi(t,x)=&2\beta \log\Big(\int_{\mathbb{R}^d}\frac{1}{(4\pi\beta (T-t))^{\frac{d}{2}}}e^{-\frac{1}{2\beta}(\frac{\|y-x_0\|^2}{2}+\frac{\|x-y\|^2}{2(T-t)})}dy\Big)\\
=& \beta d\log\frac{1}{T-t+1}-\frac{\|x-x_0\|^2}{2(T-t+1)}.
\end{split}
\end{equation*}
From \eqref{rho}, we have
\begin{equation*}
\begin{split}
\rho(t,x)=&\frac{1}{{(4\pi\beta\frac{t(T-t)}{T}})^{\frac{d}{2}}}\int_{\mathbb{R}^d}\int_{\mathbb{R}^d}\frac{e^{-\frac{1}{2\beta}(\frac{1}{2}\|z-x_0\|^2+\frac{\|x-z\|^2}{2(T-t)}+\frac{\|x-y\|^2}{2t})}}{\int_{\mathbb{R}^d}e^{-\frac{1}{2\beta}(\frac{1}{2}\|\tilde y-x_0\|^2+\frac{\|y-\tilde y\|^2}{2T})}d\tilde y}\rho_0(y) dydz\\
=&\frac{1}{{(4\pi\beta\frac{t(T-t+1)}{T+1}})^{\frac{d}{2}}}\int_{\mathbb{R}^d}e^{-\frac{1}{4\beta \frac{t(T-t+1)}{T+1}}(x-\frac{\frac{y}{t}+\frac{x_0}{T-t+1}}{\frac{1}{t}+\frac{1}{T-t+1}})^2}\rho_0(y)dy. 
\end{split}
\end{equation*}
From \eqref{kernel}, we obtain
\begin{equation*}
\begin{split}
\rho(T,x)=&\int_{\mathbb{R}^d}\frac{e^{-\frac{1}{2\beta}(\frac{1}{2}\|x-x_0\|^2+\frac{\|x-y\|^2}{2T})}}{\int_{\mathbb{R}^d}e^{-\frac{1}{2\beta}(\frac{1}{2}\|z-x_0\|^2+\frac{\|z-y\|^2}{2T})}dz}\rho_0(y)dy\\
=&\frac{1}{{(4\pi\beta\frac{T}{T+1}})^{\frac{d}{2}}}\int_{\mathbb{R}^d}e^{-\frac{1}{4\beta \frac{T}{T+1}}(x-\frac{\frac{y}{T}+x_0}{\frac{1}{T}+1})^2}\rho_0(y)dy.
\end{split}
\end{equation*}
\end{example}

\section{Regularized Wasserstein proximal operators of general energies}\label{sec4}
In this section, we study the regularized Wasserstein proximal operator of general energy functionals.  

Consider a functional $\mathcal{F}\colon \mathcal{P}_2(\mathbb{R}^d)\rightarrow\mathbb{R}$. 
We study the proximal operator of energy $\mathcal{F}$ in Wasserstein-2 space.
\begin{equation*}
\rho_T=\mathrm{WProx}_{\mathcal{F}}(\rho_0):=\arg\min_{q\in \mathcal{P}_2(\mathbb{R}^d)} \quad \mathcal{F}(q)+\frac{\mathcal{W}(\rho_0, q)^2}{2T}. 
\end{equation*}
Computing the Wasserstein proximal operator usually requires an optimization step. Instead, we study a regularized optimal control problem. 
\begin{definition}
Denote $\beta>0$ as a scalar. Consider the following mean-field control problem:
\begin{equation}\label{variation1}
\inf_{\rho, v, q} \quad \int_0^T\int_{\mathbb{R}^d}\frac{1}{2}\|v(t,x)\|^2\rho(t,x)dxdt+\mathcal{F}(q),  
\end{equation}
where the minimizer is taken among all vector fields $v\colon [0, T]\times\mathbb{R}^d\rightarrow\mathbb{R}^d$, density functions $\rho\colon [0, T)\times \mathbb{R}^d\rightarrow\mathbb{R}$ and terminal time density function $q\colon \mathbb{R}^d\rightarrow\mathbb{R}$, such that 
\begin{equation*}
\partial_t\rho(t,x)+\nabla\cdot(\rho(t,x)v(t,x))=\beta \Delta \rho(t,x),\quad \rho(0,x)=\rho_0(x), \quad \rho(T,x)=q(x).
\end{equation*}

We denote the regularized proximal operator as
\begin{equation*}
\rho_T:=\mathrm{WProx}_{T\mathcal{F},\beta}(\rho_0), 
\end{equation*}
where $\rho_T$ is the solution of density function in variational problem \eqref{variation1} at time $t=T$. 
\end{definition}
We can also represent the approximate proximal operator in a kernel representation. 
\begin{proposition}
Denote the $L^2$ first variation of functional $\mathcal{F}$ as 
\begin{equation*}
\frac{\delta}{\delta\rho}\mathcal{F}(\rho)(x)=F(x,\rho). 
\end{equation*}
Then 
\begin{equation}\label{nonlinear}
\rho(T,x)=\int_{\mathbb{R}^d}K(x,y, \beta, T, \rho_T)\rho(0,y)dy,
\end{equation}
where we define a kernel function $K\colon \mathbb{R}^d\times\mathbb{R}^d\times \mathbb{R}_+\times\mathbb{R}_+\times \mathcal{P}_2(\mathbb{R}^d)\rightarrow\mathbb{R}$ as
\begin{equation}\label{nonlinear_kernel}
K(x,y,\beta, T, \rho_T):=\frac{e^{-\frac{1}{2\beta}(F(x,\rho_T)+\frac{\|x-y\|^2}{2T})}}{\int_{\mathbb{R}^d}e^{-\frac{1}{2\beta}(F(z,\rho_T)+\frac{\|z-y\|^2}{2T})}dz}.
\end{equation}
\end{proposition}
\begin{proof}
The proof follows from those in sections \ref{sec2} and \ref{sec3}. We only need to replace $V(x)$ by $\frac{\delta}{\delta\rho}\mathcal{F}(\rho)(x)=F(x,\rho)$.
\end{proof}
\begin{remark}
We remark that for a general functional $\mathcal{F}(\rho)$, the kernel formula $K$ depends on $\rho_T$. We no longer have an explicit expression of $\rho_T$.  Instead, we make use of equation \eqref{nonlinear}. % We still need a method to compute equation \eqref{nonlinear}. 
\end{remark}
\subsection{Examples}
In this section, we present some examples of kernel representations.  
Consider an energy functional  
\begin{equation*}
\mathcal{F}(\rho)=\int_{\mathbb{R}^d} V(x)\rho(x)dx+\frac{1}{2}\int_{\mathbb{R}^d} \int_{\mathbb{R}^d}W(x-y)\rho(x)\rho(y)dxdy+\int_{\mathbb{R}^d}U(\rho(x))dx,
\end{equation*}
where $V\in C^{1}(\mathbb{R}^d; \mathbb{R})$, $W \in C^{1}(\mathbb{R}^{d}; \mathbb{R})$ is an interaction kernel function, with $W(z)=W(-z)$, $z\in \mathbb{R}^d$, and $U\in C^1(\mathbb{R}^d;\mathbb{R})$ is an internal potential function. 
The $L^2$ first variation of $\mathcal{F}(\rho)$ satisfies 
\begin{equation*}
F(x,\rho)=\frac{\delta}{\delta\rho}\mathcal{F}(\rho)(x)=V(x)+(W*\rho)(x)+U'(\rho(x)).
\end{equation*}
Thus the kernel function \eqref{nonlinear_kernel} satisfies
\begin{equation*}
K(x,y,\beta, T, \rho_T):=\frac{e^{-\frac{1}{2\beta}(V(x)+(W*\rho_T)(x)+U'(\rho_T(x))+\frac{\|x-y\|^2}{2T})}}{\int_{\mathbb{R}^d}e^{-\frac{1}{2\beta}(V(z)+(W*\rho_T)(z)+U'(\rho_T(z))+\frac{\|z-y\|^2}{2T})}dz}.
\end{equation*}
And the regularization Wasserstein proximal equation forms 
\begin{equation*}
\rho_T(x)=\int_{\mathbb{R}^d}K(x,y,\beta, T,\rho_T)\rho_0(y)dy. 
\end{equation*}
We next present some examples for equation \eqref{nonlinear} and kernel function \eqref{nonlinear_kernel}.
\begin{example}
Let $V=U=0$. Then the kernel function \eqref{nonlinear_kernel} satisfies 
\begin{equation*}
K(x,y,\beta, T, \rho_T):=\frac{e^{-\frac{1}{2\beta}((W*\rho_T)(x)+\frac{\|x-y\|^2}{2T})}}{\int_{\mathbb{R}^d}e^{-\frac{1}{2\beta}((W*\rho_T)(z)+\frac{\|z-y\|^2}{2T})}dz}.
\end{equation*}
And the regularized Wasserstein proximal equation satisfies 
\begin{equation*}
\rho_T(x)=\int_{\mathbb{R}^d}\frac{e^{-\frac{1}{2\beta}((W*\rho_T)(x)+\frac{\|x-y\|^2}{2T})}}{\int_{\mathbb{R}^d}e^{-\frac{1}{2\beta}((W*\rho_T)(z)+\frac{\|z-y\|^2}{2T})}dz}\rho_0(y)dy. 
\end{equation*}
\end{example}

\begin{example}
Let $V=W=0$. Then the kernel function \eqref{nonlinear_kernel} satisfies 
\begin{equation*}
K(x,y,\beta, T, \rho_T):=\frac{e^{-\frac{1}{2\beta}(U'(\rho_T(x))+\frac{\|x-y\|^2}{2T})}}{\int_{\mathbb{R}^d}e^{-\frac{1}{2\beta}(U'(\rho_T(z))+\frac{\|z-y\|^2}{2T})}dz}.
\end{equation*}
And the regularized Wasserstein proximal equation satisfies 
\begin{equation*}
\rho_T(x)=\int_{\mathbb{R}^d}\frac{e^{-\frac{1}{2\beta}(U'(\rho_T(x))+\frac{\|x-y\|^2}{2T})}}{\int_{\mathbb{R}^d}e^{-\frac{1}{2\beta}(U'(\rho_T(z))+\frac{\|z-y\|^2}{2T})}dz}\rho_0(y)dy.
\end{equation*}
\end{example}
\section{Numerical Experiments}\label{section5}
In this section, we conduct several computational examples to numerically verify the kernel formulations \eqref{kernel} and \eqref{nonlinear_kernel}. We also discuss a fixed point iteration algorithm that uses the kernel formulation to calculate the Wasserstein proximal operator of general functionals.

\subsection{Validation the kernel formulation with optimization method}
\subsubsection{Solve the equivalent variational problem via primal-dual hybrid gradient (PDHG) method}
To numerically verify kernel formulations, we compute the mean-field control problem \eqref{variation1} via the primal-dual hybrid gradient approach. By introducing the flux variable $m = \rho v$, we can rewrite the mean-field control problem as the following constrained optimization problem:
\begin{equation}\label{variation_flux}
\inf_{\rho, m, q} \quad \int_0^T\int_{\mathbb{R}^d}\frac{\|m(t,x)\|^2}{2\rho(t,x)}dxdt+\mathcal{F}(q),  \end{equation}
such that 
\begin{equation*}
\partial_t\rho(t,x)+\nabla\cdot m(t,x)=\beta \Delta \rho(t,x),\quad \rho(0,x)=\rho_0(x), \quad \rho(T,x)=q(x).
\end{equation*}
By introducing a Lagrangian multiplier $\Phi$, we write optimization problem \eqref{variation_flux} into the following saddle point problem:
\begin{equation} \label{eq:saddle}
    \inf_{\substack{\rho, m, q\\ \rho(0,\cdot)=\rho_0(\cdot)\\\rho(T,\cdot)=q(\cdot)}} \sup_{\Phi}\quad  \mathcal{L}(\rho,m,q,\Phi),
\end{equation}
where
\begin{equation} 
\begin{split} \label{eq:saddle_L}
\mathcal{L}(\rho, m, q, \Phi):=&\quad\int_0^T \int_{\mathbb{R}^d}\frac{\|m(t,x)\|^2}{2\rho(t,x)}dx dt+ \mathcal{F}(q)\\%\int_{\mathbb{R}^d}F(\rho(T,x))dx\\
&+\int_0^T\int_{\mathbb{R}^d} \Phi(t,x) \Big(\partial_t\rho(t,x)+\nabla\cdot m(t,x)-\beta\Delta\rho(t,x)\Big)dxdt.
\end{split}
\end{equation}

Saddle point problem \eqref{eq:saddle} is convex in $(\rho, m, q)$ and concave in $(\Phi)$. We can directly apply the primal-dual hybrid gradient method (PDHG)~\cite{champock11} to solve problem \eqref{eq:saddle}. The algorithm takes proximal updates in the primal variables $(\rho, m, q)$ and the dual variables $(\Phi)$ iteratively, with extra extrapolating steps. For more details on using the PDHG algorithm to solve mean-field control problems, we refer to \cite{Liu_mfg}. 
\subsubsection{Discretization Scheme}
{For simplicity, we consider a one dimensional finite domain $\Omega = [ -b, b]$ with the periodic boundary condition. In the following examples, we set up a large enough finite spatial domain with density $\rho$ vanishing near the boundary.}
 We use a uniform mesh for both spatial and time intervals,
 with $h_x = \frac{2b}{N_x}, h_t = \frac{T}{N_T}$, 
 for $N_x, N_T>0$, and $(t_l, x_j) = (l h_t,j h_x - b)$, for $l = 0,\cdots, N_t$, $j = 0,\cdots, N_x - 1$.
 For variables $\rho, \Phi, q$, we denote the following grid point approximation:
 \begin{align*}
     \rho_j^l = \rho(t_l,x_j),\;\quad \Phi_j^l = \Phi(t_l,x_j),\;\quad q_j = q(x_j).
 \end{align*}
 As for the flux variable $m$, we denote 
 \begin{align*}
     & m_j^l = [ m_j^{+,l} +  m_j^{-,l}], \quad m_j^{+,l} := \left(m(t_l,x_j)\right)^+, \quad  m_j^{-,l} := -\left(m(t_l,x_j)\right)^-;\\
      & \| m_j^l\|^2 = \left(m_j^{+,l} \right)^2 + \left(m_j^{-,l} \right)^2,\quad  \nabla_x \cdot m_j^l  = \frac{m_j^{+,l} - m_{j-1}^{+,l} }{h_x} +  \frac{m_{j+1}^{-,l} - m_{j}^{-,l} }{h_x}, 
 \end{align*}
where $u^+ := \max(u, 0)$ and $u^-= u^+-u$. We discretize the above variational problem with the finite difference method and obtain the following first-order scheme:
\begin{align*}
       \hat{\mathcal{L}}(\rho,m,\Phi,q)  =&\quad h_t h_x\sum_{\substack{1 \leq l \leq N_t\\  1\leq j \leq N_x}} \left( \frac{ \| m_j^{l} \|^2}{2\rho_j^l}      \right) + \hat{\mathcal{F}}(q) \\  %+F(\rho^{N_t}) \\
       & + h_t h_x\sum_{\substack{0 \leq l \leq N_t-1\\  1\leq j \leq N_x}}  \Phi^l_j \left(  \dfrac{\rho^{l+1}_j -\rho^l_j }{h_t} + {\nabla_x \cdot m_j^{l+1}} -\beta \frac{\rho^{l+1}_{j+1} + \rho^{l+1}_{j-1} -2\rho^{l+1}_{j}  }{h_x^2}\right). 
\end{align*}
More specifically, we consider the discretized energy functional $\hat{\mathcal{F}}$ as below:
\begin{align*}
\hat{\mathcal{F}}(q)  = h_x\sum_{\substack{ 1\leq j \leq N_x}} \left( V(x_j) q_j+ U\left( q_j\right)\right) + \frac{ h_x^2}{2} \sum_{\substack{ 1\leq i \leq N_x\\ 1\leq j \leq N_x}} \left( W(x_i-x_j)q_i q_j\right).
\end{align*}
Here, we denote $\rho_{M}^T(x): = \rho(T,x)$ and use it to compare with the solution obtained from the kernel formulation.

As for computation using kernel formulation, we use the Riemann sum.
Taking the linear energy functional as an example, i.e., $\mathcal{F}(\rho)=\int_{\Omega} V(x)\rho(x)dx$, the kernel function in equation~\eqref{kernel} can be approximated by the following form:
\begin{equation*}\label{kernel_riemann}
K(x_i,y_j,\beta, T,V) = \frac{e^{-\frac{1}{2\beta}(V(x_i)+\frac{\|x_i-y_j\|^2}{2T})}}{ h_x\sum_{\substack{ 1\leq k \leq N_x}} e^{-\frac{1}{2\beta}(V(z_k)+\frac{\|z_k-y_j\|^2}{2T})}}.
\end{equation*}
Hence the solution $\rho(T,x_i)$ via kernel integral formula can be expressed  as
\begin{equation*}
\rho(T,x_i)=  h_x\sum_{\substack{ 1\leq j \leq N_x}} K(x_i,y_j, \beta, T, V)\rho(0,y_j).
\end{equation*}
From now on, we denote $\rho_{K}^T(x): = \rho(T,x)$. 
\subsubsection{Example A}\label{subsec:eg1}
In the first example, we consider the regularized Wasserstein proximal operator of linear energy $\mathcal{F}(\rho)=\int_{\Omega} V(x)\rho(x)dx$.
For example parameters, we have the following setup:
\begin{align*}
   & b=5, T = 0.2, \beta =0.25,\\
    &    V(x) =  e^{-\frac{(x+0.25)^2}{0.5}},\\
     &   \rho_0(x) = \frac{1}{\sigma_0 \sqrt{2\pi}} e^{-\frac{\left(x-0.25 \right)^2}{2\sigma_0^2}}, \;\sigma_0 = 0.1.
\end{align*}
By applying different meshes $N_x, N_t$, we record the difference between $\rho_{M}^T$ (obtained via the optimization approach) and $\rho_K^T$ (obtained via the kernel formula).
From Table~\ref{tab:linear_convergence}, we observe the consistent convergence as refining the mesh grid, which suggests that our kernel formula is precise when the energy functional is of linear form. In Figure~\ref{fig:example_a_linear}, we show the numerical solutions of Example A.
\begin{table}[!htp]\centering
% \scriptsize
\begin{tabular}{lrrrr}\toprule
$h_x$ &$0.0625$ &$0.05$ &$0.0313$ \\
\midrule
$\|\rho_K^T - \rho_M^T\|_{L^2}$ &$0.0016$ &$0.0011$ &$0.0007$ \\
\bottomrule
\end{tabular}
\caption{Record of $\|\rho_K^T - \rho_M^T\|_{L^2}$ with different meshes of Example A.}\label{tab:linear_convergence}
\end{table}
	\begin{figure}[htbp!]
		\includegraphics[width=0.8\textwidth,trim=50 0 50 0, clip=fale]{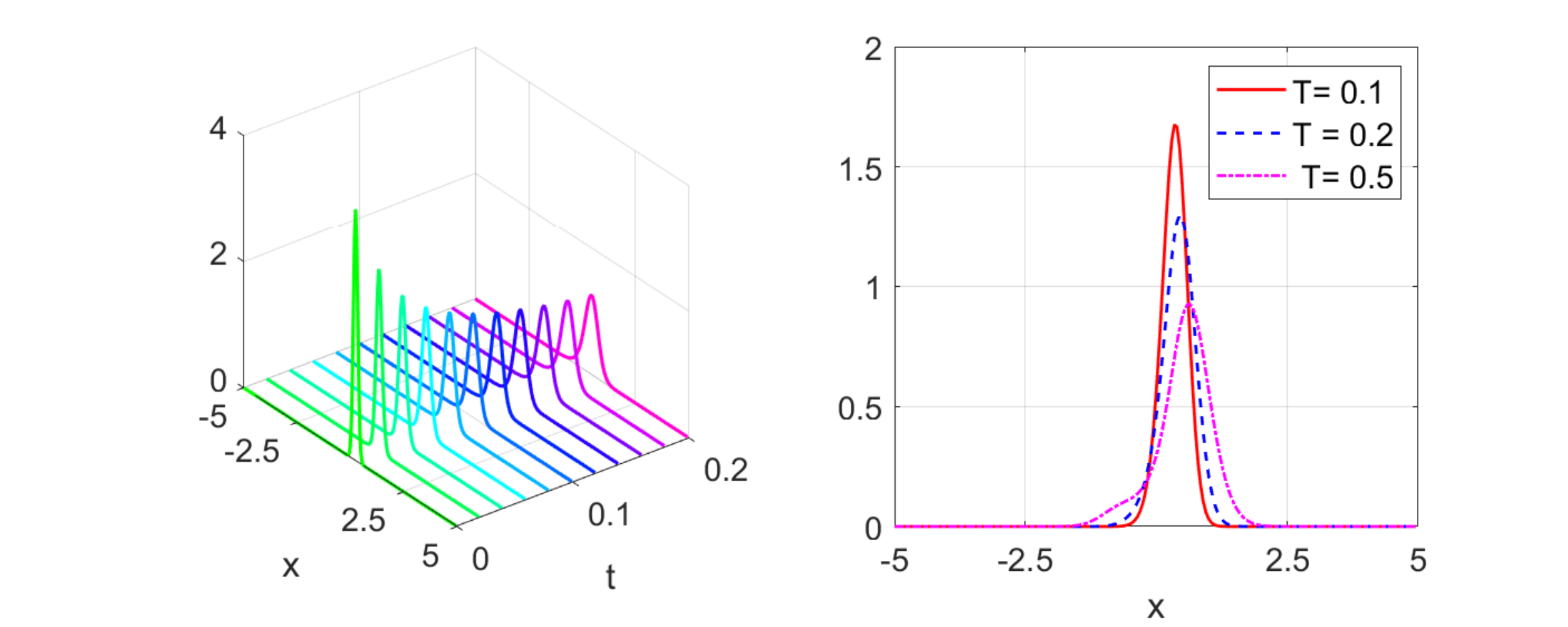}
  % trim={<left> <lower> <right> <upper>}
		\caption{Numerical results for Example~A. On the left, we show the time evolution of $\rho(t,x)$ over the time interval $[0,0.2]$. On the right, we vary the coefficient $T = 0.1, 0.2, 0.5$ to see how the parameter $T$ affects the regularized Wasserstein proximal operator and its resulting density distribution.}
		\label{fig:example_a_linear}
	\end{figure}

\subsubsection{Example B}\label{subsec:eg2}
We now consider $\mathcal{F}(\rho)=\frac{1}{2}\int_{\Omega} \int_{\Omega}W(x-y)\rho(x)\rho(y)dxdy$, where the energy functional is quadratic on $\rho$. Specifically, we have
\begin{align*}
   & b=5, T = 0.1, \beta =0.1,\\
    &    W(x-y) = \lambda_W (x-y)^2,\; \lambda_W = 0.2,\\
     &    \rho_0(x) = \frac{1}{\sigma_0 \sqrt{2\pi}} e^{-\frac{\left(x-0.5 \right)^2}{2\sigma_0^2}},\;\sigma_0 = 0.1 .
\end{align*}
With mesh $N_x$, we first solve the corresponding variational problem and obtain $\rho_M^T$. Then we apply the kernel formulation to obtain \begin{align*}
  \rho_K^T(x) = \int_{\Omega} K(x,y,\beta,T,\rho_M^T) \rho_0(y)dy.
\end{align*}
By varying the mesh, we obtained the difference $\|\rho_K^T - \rho_M^T\|_{L^2}$ converges to $0$ with the order $O(h_x)$. See Table~\ref{tab:nonlocal_convergence}. This observation suggests that the kernel formulation could be applied to the computation of Wasserstein proximal of interaction energy. 
\begin{table}[!htp]\centering
% \scriptsize
\begin{tabular}{lrrrr}\toprule
$h_x$ &$0.0625$ &$0.05$ &$0.0313$ \\
\midrule
$\|\rho_K^T - \rho_M^T\|_{L^2}$ &$6.3e-4$	&$ 4.6e-4$	& $3.1e-4$ \\
\bottomrule
\end{tabular}
\caption{Record of $\|\rho_K^T - \rho_M^T\|_{L^2}$ with different meshes of Example B.} \label{tab:nonlocal_convergence}
\end{table}

\subsubsection{Example C}\label{subsec:eg3}
In this example, we consider the energy being the Kullback--Leibler (KL) divergence, i.e.,  $\mathcal{F}(\rho) = \lambda_F \int_{\Omega} \rho(x)\log(\frac{\rho(x)}{\rho_F(x)}) dx$. In detail, the example parameters are as follows:
\begin{align*}
   & b=5, T = 0.2, \beta =0.1, \lambda_F = 0.1,\\
    &    \rho_F(x) = \frac{1}{\sigma_F \sqrt{2\pi}} e^{-\frac{x^2}{2\sigma_F^2}}, \;\sigma_F = 0.4,\\
     &   \rho_0(x) = \frac{1}{\sigma_0 \sqrt{2\pi}} e^{-\frac{\left(x-0.5 \right)^2}{2\sigma_0^2}}, \;\sigma_0 = 0.2 .
\end{align*}
To avoid numerical instability, we also modify the KL divergence with
\begin{align*}
\mathcal{F}(\rho) = \lambda_F \int_{\Omega} \rho(x)\log(\frac{\rho(x) +\epsilon}{\rho_F(x) + \epsilon}) dx, \; \epsilon = 1e-4.
\end{align*}

Similarly to Example~B, we first solve the corresponding mean-field control problem with solution $\rho_M^T$. Via kernel formulation, we arrive at $ \rho_K^T(x) = \int_{\Omega} K(x,y,\beta,T,\rho_M^T) \rho_0(y)dy$.
By refining the mesh, we observe the consistent convergence of numerical solutions. See Table~\ref{tab:KL_convergence}.
This numerical result also suggests that, with proper regularization parameters, the kernel formula of the regularized Wasserstein proximal operator is valid for general functions. We present the regularized Wasserstein proximal and the time-evolution of $(\rho(t,\cdot), \Phi(t,\cdot))$ in Figure~\ref{fig:example_c_KL}.

\begin{table}[!htp]\centering
% \scriptsize
\begin{tabular}{lrrrr}\toprule
$h_x$ &$0.0625$ &$0.05$ &$0.0313$ \\
\midrule
$\|\rho_K^T - \rho_M^T\|_{L^2}$ &$2.5e-4$ &$1.7e-4$ &$7.8e-5$ \\
\bottomrule
\end{tabular}
\caption{Record of $\|\rho_K^T - \rho_M^T\|_{L^2}$ with different meshes of Example C.}\label{tab:KL_convergence}
\end{table}

	\begin{figure}[htbp!]
		\includegraphics[width=0.8\textwidth,trim=50 0 50 0, clip=fale]{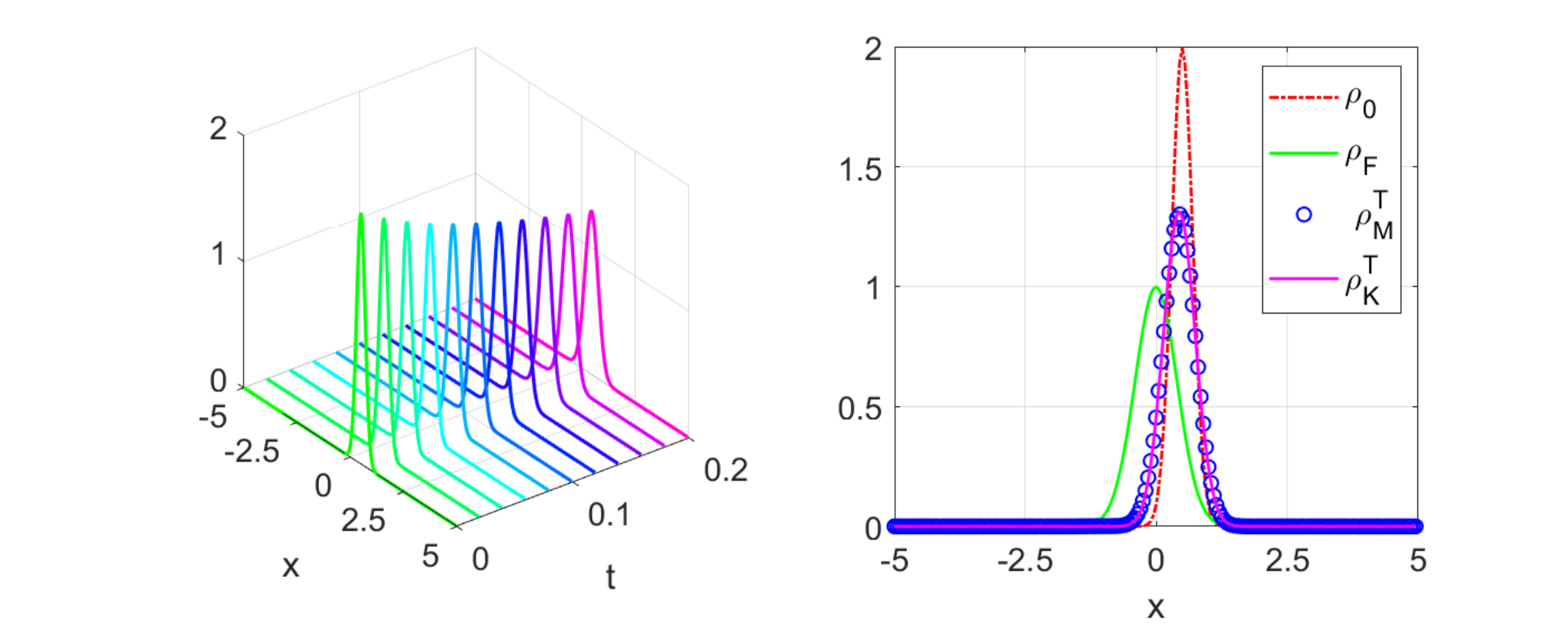}
  % trim={<left> <lower> <right> <upper>}
		\caption{Numerical results for Example~C. On the left, we show the time evolution of $\rho(t,x)$ over the time interval $[0,0.2]$. On the right, we plot the example parameters $\rho_0, \rho_F$, the mean-field control problem solution $\rho_M^T$, and the density $\rho_K^T(x)$ obtained via the kernel formulation.}
		\label{fig:example_c_KL}
	\end{figure}

\subsection{A fixed-point type iterative approach to compute the regularized Wasserstein proximal}
For general functional $\mathcal{F}(\rho)$, the kernel formulation can be modified as a fixed point map, i.e.,
\begin{align}\label{scheme:fix_point}
    \rho_{itr}^{n}(x)=\int_{\Omega}K(x,y,\beta, T,\rho_{itr}^{n-1})\rho_0(y)dy, \quad n = 1,2,...,N. 
\end{align}
Such an iteration scheme can efficiently calculate the regularized Wasserstein proximal operator.

We provide a numerical example to show that, with proper regularization ($\beta>0$), the scheme~\eqref{scheme:fix_point} is a fixed point iteration. Moreover, its limit is the solution to the regularized Wasserstein proximal.

We first consider the regularized Wasserstein proximal from Example~C in section \ref{subsec:eg3}. To apply the fixed point iteration, we fix the mesh size $h_x$ and set $\rho_{itr}^0 = \rho_0$ as the initialization.
Then we apply the scheme~\eqref{scheme:fix_point} via a Riemann summation. To check if the fixed point iteration limit is the regularized Wasserstein proximal, we compute the scheme~\eqref{scheme:fix_point} with different mesh $N_x$ and compare it with the solution of the corresponding variational problem. Table~\ref{tab:fixed_pt} suggests that the limit point of the iteration~\eqref{scheme:fix_point} matches nicely. We also present the iteration details in Figure~\ref{fig:example_fixpoint}, which shows the fast convergence of the fixed point iteration.

\begin{table}[!htp]\centering
% \scriptsize
\begin{tabular}{lrrrr}\toprule
$h_x$ &$0.0625$ &$0.05$ &$0.0313$ \\
\midrule
$\|\rho_{itr}^{20} - \rho_M^T\|_{L^2}$ &$1.2e-2$ &$9.5e-3$ &$6.4e-3$ \\
\bottomrule
\end{tabular}
\caption{Record of $\|\rho_{itr}^{20} - \rho_M^T\|_{L^2}$ with different meshes of Example 3 using the fixed-point iteration.}\label{tab:fixed_pt}
\end{table}

\begin{figure}[htbp!]
    \includegraphics[width=0.8\textwidth,trim=50 0 50 0, clip=fale]{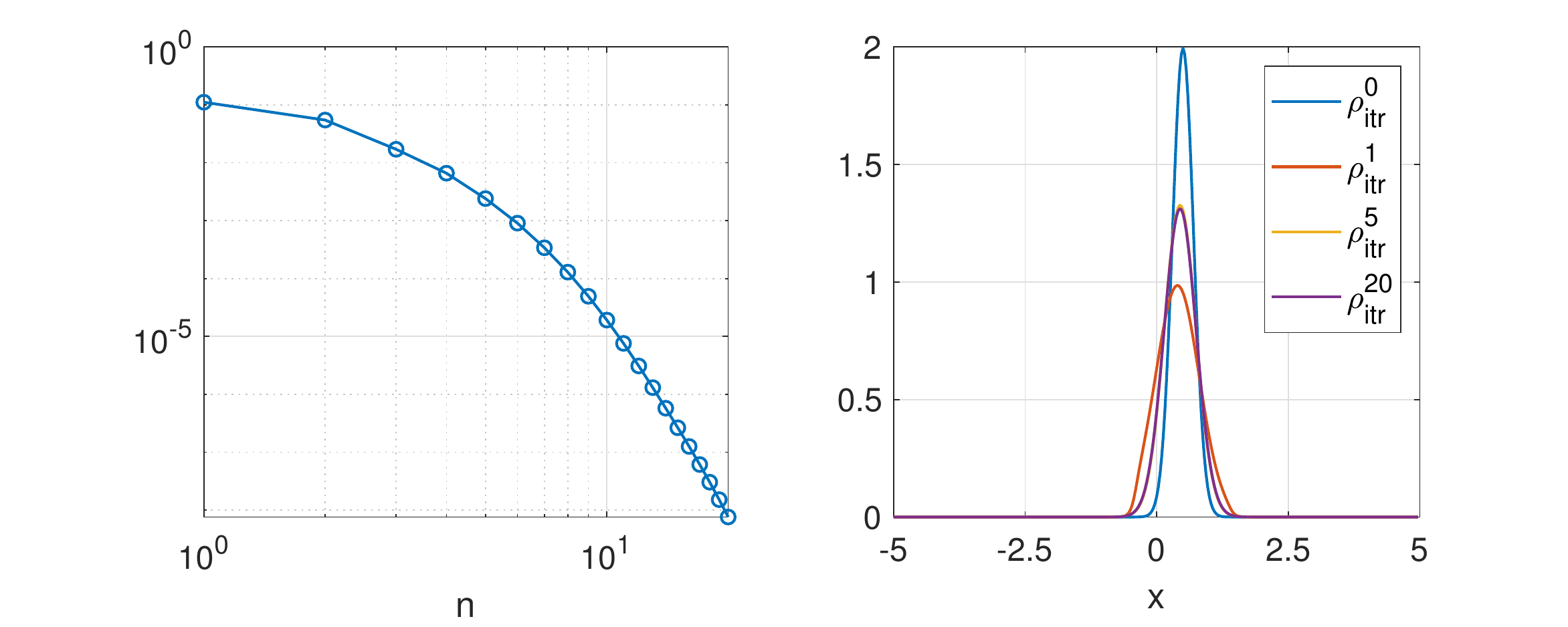}
  % trim={<left> <lower> <right> <upper>}
\caption{Numerical results for the fixed-point iteration. On the left, we plot the difference in two adjacent iterations, i.e.,  $\|\rho_{itr}^n - \rho_{itr}^{n-1}\|_{L^2}$ for the first $20$ iterations. On the right, we plot some intermediate points of the iteration process. We observe that at the $5$th iteration, the density profile is already very close to the limit in the $L^2$ norm.}
\label{fig:example_fixpoint}
\end{figure}

Let us consider another example with $\mathcal{F}(\rho) = 0.5\int_\Omega\rho(x)^2dx$, and apply the fixed-point iteration \eqref{scheme:fix_point}. For other parameters, we set
\begin{align*}
   & b=5, T = 1, \beta =0.25,\\
     &    \rho_0(x) = \frac{1}{\sigma_0 \sqrt{2\pi}} e^{-\frac{\left(x-0.5 \right)^2}{2\sigma_0^2}},\;\sigma_0 = 0.1 .
\end{align*}
Numerical results are shown in Figure~\ref{fig:example_fixpoint_d}. We also compare the fixed-point iterations using the kernel formula with the density $\rho_M^T$ obtained by solving the corresponding variational problem \eqref{variation_flux}. These solutions match pretty well after $40$ steps of the fixed-point iteration. 
\begin{figure}[htbp!]
    \includegraphics[width=0.8\textwidth,trim=50 0 50 0, clip=fale]{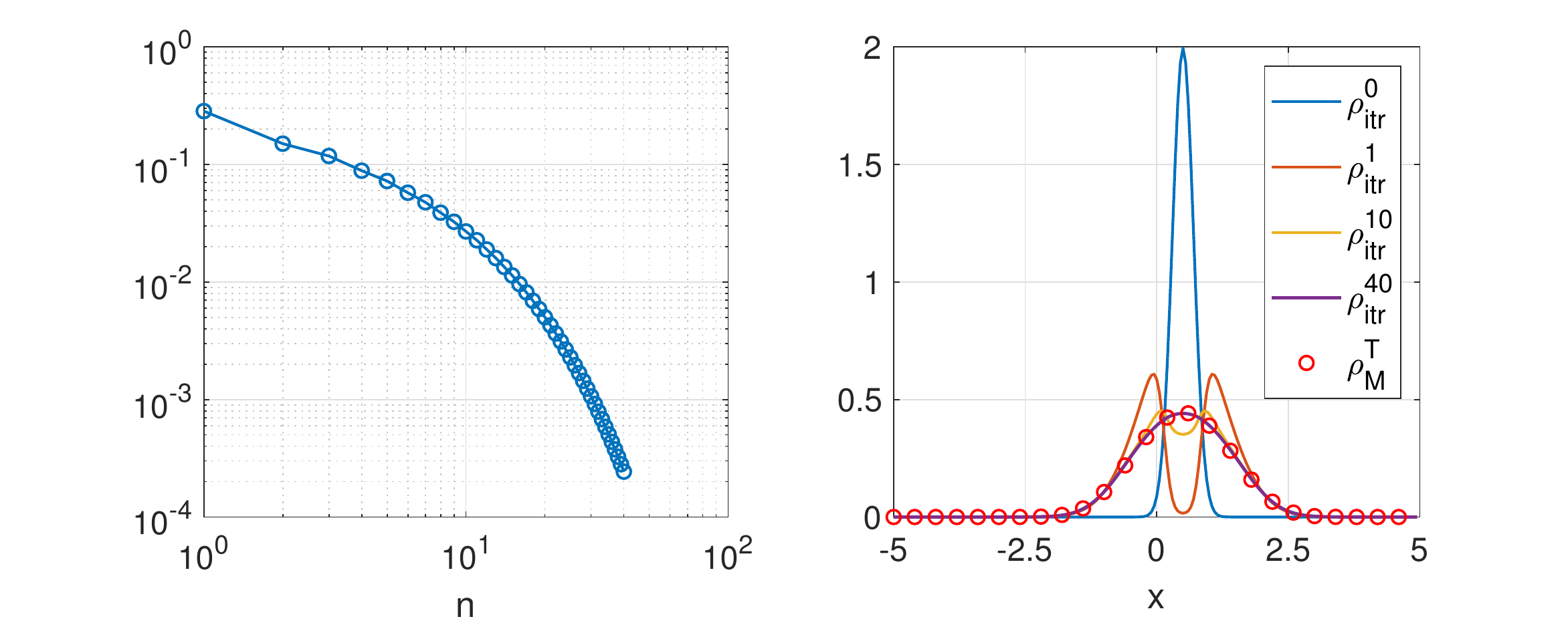}
  % trim={<left> <lower> <right> <upper>}
\caption{Numerical results for the fixed-point iteration of functional $\mathcal{F}(\rho) = 0.5 \int_\Omega \rho(x)^2dx$. In the left figure, we plot the difference between two adjacent iterations. In the right figure, we plot the density profiles of some intermediate iteration processes and observe that the limit of the fixed point scheme converges to $\rho_M^T$.}
\label{fig:example_fixpoint_d}
\end{figure}

\bibliographystyle{abbrv}

\begin{thebibliography}{10} 
\bibitem{AGS}
L.~Ambrosio, N.~Gigli and G.~Savar{\'e}.
\newblock{\em Gradient Flows in Metric Spaces and in the Space of Probability Measures.}
\newblock{Lectures in Mathematics ETH Zurich}, Birkhauser Verlag, Basel, 2nd ed. 2008.

\bibitem{BCGL}
\newblock{J.~Backhoff, G.~Conforti, I.~Gentil, and C.~Leonard}.
\newblock{The mean field Schr{\"o}dinger problem: ergodic behavior,
entropy estimates and functional inequalities.}
\newblock{\em Probability Theory and Related Fields}, 178:475–530, 2020.

\bibitem{BB}
\newblock{J.D. Benamou and Y. Brenier.}
\newblock{A Computational Fluid Mechanics Solution to the Monge-Kantorovich
Mass Transfer Problem.}
\newblock{\em Numerische Mathematik}, 84(3):3750--393, 2000.

\bibitem{B}
E.~Bernton.
\newblock{Langevin Monte Carlo and JKO splitting.}
\newblock{\em Proceedings of the 31st Conference On Learning Theory}, PMLR 75:1777-1798, 2018.

\bibitem{BGU}
T.~Bonnemain,  T.~Gobron, and  D.~Ullmo. 
\newblock{Lax connection and conserved quantities of quadratic mean field games.}
\newblock{\em J. Math. Phys}, 62, 083302, 2021. 



\bibitem{champock11}
A. Chambolle and T. Pock.
\newblock {A first-order primal-dual algorithm for convex problems with applications to imaging}.
\newblock{\em J. Math. Imaging Vision.}, 40(1):120-145, 2011.

\bibitem{COO}
P.~Chaudhari, A.~Oberman, S.~Osher, S.~Soatto and G.~Carlier. 
\newblock{Deep relaxation: partial differential equations for optimizing deep neural networks.}
\newblock{\em Res Math Sci}, 5, 30, 2018.

\bibitem{CGP}
Y.~Chen, T.~Georgiou, and M.~Pavon.
\newblock{On the Relation Between Optimal Transport and
Schr{\"o}dinger Bridges: A Stochastic Control Viewpoint.}
\newblock{\em J Optim Theory Appl}, 169:671--691, 2016.

\bibitem{Conforti2018_second}
G.~Conforti.
\newblock A second order equation for {{Schr\"odinger}} bridges with
  applications to the hot gas experiment and entropic transportation cost.
\newblock {\em Probability Theory and Related Fields}, 174, 1–47, 2019.

\bibitem{D4}
Y.~Dai, Y.~Jiao, L.~Kang, X.~Lu, and J.~Yang.
\newblock{Global Optimization via Schr{\"o}dinger-F{\"o}llmer Diffusion.}
\newblock{\em arXiv:2111.00402}, 2021. 


\bibitem{fol88}
H.~F{\"o}llmer.
\newblock Random fields and diffusion processes.
\newblock In {\em {\'E}cole d'{\'E}t{\'e} de Probabilit{\'e}s de Saint-Flour
  XV--XVII, 1985--87}, volume 1362 of {\em Lecture Notes in Math.}, pages
  101--203. Springer, Berlin, 1988.

\bibitem{OG}
O.~Guéant.
\newblock{Mean field games equations with quadratic Hamiltonian: a specific approach. }
\newblock{\em Mathematical Models and Methods in Applied Sciences}, Vol. 22, No. 09, 1250022, 2012. 

\bibitem{OG1}
O.~Guéant, J.M.~Lasry, and P.L.~Lions.  \newblock{Mean Field Games and Applications.} 
\newblock{\em In: Paris-Princeton Lectures on Mathematical Finance}, 2010.


\bibitem{HFO}
H.~Heaton, S.W.~Fung, and S.~Osher. 
\newblock{Global Solutions to Nonconvex Problems by Evolution of Hamilton-Jacobi PDEs.}
\newblock{\em arXiv:2202.11014}, 2022. 

\bibitem{JKO}
R.~Jordan, D.~Kinderlehrer, and F.~Otto.
\newblock{The Variational Formulation of the Fokker--Planck Equation.}
\newblock{\em SIAM Journal on Mathematical Analysis}, Vol. 29, Iss. 1, 1998.

\bibitem{LL}
J.~M.~Lasry, and P.~L.~Lions.
\newblock{Mean field games.}
\newblock{\em Japanese Journal of Mathematics}, 2, 229-260, 2007.

\bibitem{Liu_mfg}
S.~Liu, M.~Jacobs, W.~Li, L.~Nurbekyan, and S. Osher.
\newblock{Computational Methods for First-Order Nonlocal Mean Field Games with Applications}
\newblock{\em SIAM Journal on Numerical Analysis}, Vol. 59, 2639-2668,2021.

\bibitem{LeL}
F.~Leger, and W.~Li.
\newblock{Hopf--Cole transformation via generalized Schr{\"o}dinger bridge problem.}
\newblock{\em Journal of Differential Equations}, Volume 274, 788--827, 2021.


\bibitem{Leonard2013_surveya}
C.~L\'eonard.
\newblock A survey of the {{Schr\"odinger}} problem and some of its connections
  with optimal transport.
\newblock {\em Discrete and Continuous Dynamical Systems}, 34(4):1533--1574,
  2013.

\bibitem{LLW}
W.~Li, J.~Lu, and L.~Wang.
\newblock{Fisher information regularization schemes for Wasserstein gradient flows.}
\newblock{\em Journal of Computational Physics}, Volume 416, 109449, 2020. 


\bibitem{WGAN}
A.~Lin, W.~Li, S.~Osher, and G.~Montufar.
\newblock{Wasserstein Proximal of GANs.}
\newblock{\em  Geometric Science of Information}, pp 524-533, 2021. 

\bibitem{Nelson2}
E.~Nelson.
\newblock {\em Quantum {{Fluctuations}}}.
\newblock Princeton series in physics. {Princeton University Press}, Princeton,
  N.J, 1985.


\bibitem{OHF}
S.~Osher, H.~Heaton, and S.W.~Fung. 
\newblock{A Hamilton-Jacobi-based Proximal Operator.}
\newblock{\em arXiv:2211.12997}, 2022. 


\bibitem{Villani2009_optimal}
C.~Villani.
\newblock {\em Optimal Transport: Old and New}.
\newblock Number 338 in Grundlehren Der Mathematischen {{Wissenschaften}}.
  {Springer}, Berlin, 2009.

\bibitem{Yasue1981_stochastica}
K.~Yasue.
\newblock Stochastic calculus of variations.
\newblock {\em Journal of Functional Analysis}, 41(3):327--340, 1981.

\bibitem{Zambrini1986}
J. C. Zambrini.
\newblock Variational processes and stochastic versions of mechanics.
\newblock {\em J. Math. Phys.}, 27(9):2307--2330, 1986.

\end{thebibliography}

\end{document}